\documentclass{amsart}
\usepackage[T1]{fontenc}

\usepackage{amssymb}
\usepackage{amssymb,amsthm}

\usepackage{amscd,amssymb,graphics}

\usepackage[usenames]{color}
\usepackage{mathrsfs}

\usepackage{color}
\usepackage{amsfonts}
\usepackage{amsmath}
\usepackage{amsxtra}
\usepackage{latexsym}
\usepackage[mathcal]{eucal}
\usepackage{enumerate}
\usepackage{ifsym}
\usepackage{yfonts}
\usepackage{calc}
\usepackage{MnSymbol}
\usepackage[numbers]{natbib}
\usepackage[left=1in,right=1in,top=1in,bottom=0.98in]{geometry}
\input xy
\xyoption{all}

\newtheorem{theorem}{Theorem}[section]
\newtheorem*{theoremA}{Theorem A}
\newtheorem*{theoremB}{Theorem B}
\newtheorem*{theoremC}{Theorem C}
\newtheorem{fact}[theorem]{Fact}
\newtheorem{corollary}[theorem]{Corollary}
\newtheorem{lemma}[theorem]{Lemma}
\newtheorem{proposition}[theorem]{Proposition}

\newtheorem{question}[theorem]{Question}

\theoremstyle{definition}
\newtheorem{definition}[theorem]{Definition}
\theoremstyle{remark}

\newtheorem{example}[theorem]{Example}

\newcommand{\ben}{\begin{enumerate}}
\newcommand{\een}{\end{enumerate}}
\newcommand{\bit}{\begin{itemize}}
\newcommand{\eit}{\end{itemize}}

\def\R {{\mathbb R}}
\def\Q {{\mathbb Q}}

\def\N{{\mathbb N}}
\def\T{{\mathbb T}}

\def\Z {{\mathbb Z}}

\def\H{{\mathcal H}}

\def\QED{\nobreak\quad\ifmmode\roman{Q.E.D.}\else{\rm Q.E.D.}\fi}

\def\HLM{hereditarily  locally minimal}
\def\H(L)M{hereditarily  (locally) minimal}
\def\D(L)M{densely (locally) minimal}
\def \HM {hereditarily   minimal}
\def\DLM{densely locally minimal}
\def\DM{densely minimal}
\def\DTM{densely totally minimal}

\def \HM {hereditarily   minimal}

\DeclareMathOperator{\SO}{SO}
\DeclareMathOperator{\SL}{SL}

\begin{document}

\title{Densely locally minimal groups}
\author[Xi]{W. Xi}
\address[W. Xi]{\hfill\break
	School of Mathematical Sciences
	\hfill\break
	Nanjing Normal University
	\hfill\break
	Wenyuan Road No. 1, 210046 Nanjing
	\hfill\break
	China}
\email{xiwenfei0418@outlook.com}

\author[Dikranjan]{D. Dikranjan}
\address[D. Dikranjan]{\hfill\break
Dipartimento di Matematica e Informatica
\hfill\break
Universit\`{a} di Udine
\hfill\break
Via delle Scienze  206, 33100 Udine
\hfill\break
Italy}
\email{dikranja@dimi.uniud.it}
\author[Shlossberg]{M. Shlossberg}
\address[M. Shlossberg]{\hfill\break
Dipartimento di Matematica e Informatica
\hfill\break
Universit\`{a} di Udine
\hfill\break
Via delle Scienze  206, 33100 Udine
\hfill\break
Italy}
\email{menachem.shlossberg@uniud.it}
\author[Toller]{D. Toller}
\address[D. Toller]{\hfill\break
Dipartimento di Matematica e Informatica
\hfill\break
Universit\`{a} di Udine
\hfill\break
Via delle Scienze  206, 33100 Udine
\hfill\break
Italy}
\email{daniele.toller@uniud.it}



\begin{abstract}
We study locally compact groups having all dense subgroups (locally) minimal. We call such groups {\it densely} {\it (locally) minimal}. In 1972 Prodanov proved that the infinite compact abelian groups having all subgroups minimal are precisely the groups $\Z_p$ of $p$-adic integers. In \cite{firstpaper}, we extended Prodanov's theorem to the non-abelian case at several levels. In this paper, we focus on the densely (locally) minimal abelian groups.

We prove that in case that a topological abelian group $G$ is either compact or connected locally compact, then $G$ is densely locally minimal if and only if $G$ either is a Lie group or has an open subgroup isomorphic to $\Z_p$ for some prime $p$. This should be compared with the main result of \cite{DHXX}. Our Theorem C provides another extension of Prodanov's theorem: an infinite locally compact group is densely minimal if and only if it is isomorphic to $\Z_p$. In contrast, we show that there exists a densely minimal, compact, two-step nilpotent group that neither is a Lie group nor it has an open subgroup isomorphic to $\Z_p$.
\end{abstract}

\maketitle

\begin{center}
\today
\end{center}

\section{Introduction}
A Hausdorff group $(G, \tau)$ is called {\it minimal} if there exists no Hausdorff group topology on $G$ which is strictly coarser than $\tau$. Equivalently,  if every continuous isomorphism $G\longrightarrow H$, with a Hausdorff group $H$, is a topological isomorphism.
Introducing independently this notion,  Do\"{\i}tchinov and  Stephenson \cite{Doitch,S71} also provided the first examples of minimal groups which are not compact. For more information on minimal groups we refer the reader to  \cite{B,DS,EDS,P} (see also the survey \cite{DMe} and the book \cite{DPS}).
Note that locally compact groups need not be minimal. Actually, Stephenson proved the following.
\begin{fact}\label{fact:step}
\cite[Theorem 1]{S71} \label{Steph:Thm} A minimal locally compact abelian group is compact.
\end{fact}

Morris and Pestov \cite{MP} introduced the following class of groups which contains all minimal groups and all locally compact groups. A Hausdorff  group $(G, \tau)$ is called {\it locally minimal}  if there exists a neighborhood $V$ of the identity of $G$, such that for every coarser Hausdorff group topology $\sigma\subseteq \tau$ with $V\in \sigma$ one has $\sigma=\tau$.

Neither minimality nor local minimality is inherited by all subgroups.
For this reason, we studied in \cite{firstpaper} the locally compact groups having all subgroups (locally) minimal, introducing the following terminology.
\begin{definition}\cite[Definition 1.2]{firstpaper}\label{hmhlm}
A topological group $G$ is said to be {\em hereditarily} {\em (locally)}  {\em minimal}, if every subgroup of $G$ is  (locally) minimal.
\end{definition}

 Prodanov proved that the $p$-adic integers are the only infinite \HM\ compact groups, namely:
\begin{fact}\label{TeoP}\cite{P}\label{fac:HMZ}
An infinite compact abelian group $K$ is isomorphic to $\Z_p$ for some prime $p$ if and only if  $K$ is \HM.
\end{fact}

In \cite{firstpaper} we extended this result to non-abelian groups at various levels of non-commutativity (in particular, see the solvable case \cite[Theorem D]{firstpaper}). As a starting point of our current research consider the following extension of Prodanov's theorem which unifies Lie theory and $p$-adic numbers ``under the same umbrella''.

\begin{fact}{\cite[Corollary 1.11]{DHXX}}\label{fac:ext}
For a locally compact  group $K$ that is either abelian or connected the following conditions are equivalent:
\ben[(a)]
\item $K$ is a \HLM \ group;
\item $K$ either is a Lie group or has an open subgroup isomorphic to $\Z_p$ for some prime $p.$\een
\end{fact}

In this paper, we study the topological groups having all {\em dense} subgroups (locally) minimal. We mainly focus on locally compact abelian groups.
\begin{definition}\label{dmdlm}
A topological group $G$ is said to be {\em densely (locally) minimal}, if every dense subgroup of $G$ is  (locally) minimal.
\end{definition}

Here are some examples of compact abelian groups that are not densely locally minimal.
\begin{example} \label{ex:DHM:groups} \

\begin{enumerate}
\item
The Pontryagin dual $\widehat{\Q}$ of the discrete group $\Q$ is monothetic, i.e., $\widehat{\Q} = \overline{ C }$ for some (infinite) cyclic subgroup $C$ of $\widehat{\Q}$. Obviously $C$ is (algebraically) isomorphic to $\Z$, but it does not carry the $p$-adic topology for any prime $p$. So $C$ is not minimal, since the $p$-adic topologies are the only minimal topologies on $\Z$ (see for example, \cite[page 5]{DMe}). Moreover, as $\widehat{\Q}$ is a compact torsion-free abelian group we deduce that $C$ is not even locally minimal by \cite[Corollary 4.7]{ACDD}.

\item
By \cite[Lemma 3.1]{DHXX}, $\Z_{p}^{2}$ has a dense subgroup that fails to be locally minimal.

\item
If $p\ne q$ are primes, then $\Z_{p} \times \Z_{q}$ is not densely locally minimal by \cite[Lemma 3.2]{DHXX}.
\end{enumerate}
\end{example}

In fact, Theorem A,  which generalizes in part Fact  \ref{fac:ext}, characterizes all compact abelian groups that are \DLM. This answers {\cite[Question 5.3]{DHXX}} in the positive for compact groups.

\begin{theoremA}
For a compact abelian group $K$, the following conditions are equivalent:
\ben[(a)]
\item $K$ is a \HLM \ group;
\item $K$ is a \DLM \ group;
\item $K$ either is a Lie group or has an open subgroup isomorphic to $\Z_p$ for some prime $p.$\een
\end{theoremA}

 For compact two-step nilpotent groups  conditions $(a)$ and $(b)$ of Theorem A are not equivalent (see Example \ref{ex:ac}).

 \medskip

 Theorem A and the next result, which provides another connection to Fact  \ref{fac:ext}, are both  proved in Section \ref{sec:thmA} by using some preliminary results from Section \ref{sec:pr}.
\begin{theoremB}
	For a connected locally compact abelian group $G$, the following conditions are equivalent:
	\ben[(a)]
	\item $G$ is a \HLM \ group;
	\item $G$ is a \DLM \ group;
	\item $G$ is a Lie group.\een
\end{theoremB}

In case a topological abelian group  is either compact or connected locally compact, then it is \DLM\ if and only if it is \HLM\ by Theorem A and Theorem B. Negatively answering {\cite[Question 5.3]{DHXX}}, we now show  that this equivalence is not true in general.

\begin{example}\label{ex:nab}
Obviously, if a locally minimal group has no proper dense subgroups, it is \DLM. In particular, this implication is true for locally compact groups.

Khan \cite[Proposition 5.1]{KH} proved that a non-discrete locally compact abelian group $G$ has no proper dense subgroups if and only if $G$ is totally disconnected, every compact subgroup of $G$ is torsion, and $pG$ is an open subgroup of $G$ for every prime $p$.
In  particular, such a group neither is a Lie group nor it contains a copy of $\Z_p$, for any prime $p$. By Fact \ref{fac:ext},  $G$ is not \HLM.

For a concrete example, consider the following construction. Equip $K =\Z(2)^\omega $ with the usual product topology, and consider the group $G=  \Z(2^{\infty}) ^\omega$, equipped with the smallest group topology having $K$ open.
Then $G$ is non-discrete, locally compact, divisible, abelian group, and one can check that $G$ has no proper dense subgroups (see for example {\cite[\S 7.2]{HBSTT}}).
\end{example}

Apart from taking subgroups, the class of minimal groups is also not stable under taking Hausdorff quotients, so Dikranjan and Prodanov \cite{DP} introduced the following stronger notion: a topological group is \emph {totally minimal} if all of its Hausdorff quotients are minimal. In this paper, we also consider the topological groups having all dense subgroups totally minimal.
\begin{definition}
A topological group $G$ is said to be {\em densely totally minimal}, if every dense subgroup of $G$ is totally minimal.
\end{definition}
By the Total Minimality Criterion (see Fact \ref{Crit}(3)), a group is densely totally minimal if and only if it is totally minimal and every dense subgroup is totally dense. This implies the following reformulation of T. Soundararajan \cite{ts} (see  also \cite[Theorem 3.1]{CD0} and \cite[Exercise 5.5.6]{DPS}).
\begin{fact}
An infinite compact abelian group that is densely totally minimal is isomorphic to $\Z_p$ for some prime $p$.
\end{fact}

As a densely totally minimal group is obviously densely minimal, we extend the above fact as follows.
\begin{theoremC}
For an infinite locally compact abelian group $K$, the following conditions are equivalent:
\ben[(a)]
\item $K$ is a \HM \ group;
\item $K$ is a \DM \ group;
\item $K$ is isomorphic to $\Z_p$ for some prime $p.$\een
\end{theoremC}

We prove Theorem C in Section \ref{sec:vs}, where  dense minimality  and hereditary minimality  are  compared in general.
In particular, we see that these properties are not equivalent even for  compact two-step nilpotent groups (see Example \ref{ex:ac}).

In \cite[Proposition 3.9]{firstpaper}, we proved that a \HM\ locally compact group is totally disconnected,  while Example \ref{ex:SO} shows that there exist \DTM\ compact groups that are pathwise connected.

\smallskip
Section \ref{sec:open} collects some open questions and final remarks. Furthermore, Proposition \ref{prop:hscon} and Question \ref{q:HSC} deal with the Hilbert-Smith Conjecture.
	
\bigskip

The next diagram describes some of the interrelations between the properties considered so far. The concepts, which are introduced in Definition \ref{hmhlm} and Definition \ref{dmdlm}, are abbreviated here to HM, HLM, DM and DLM. The double arrows denote implications that
always hold.
The single arrows denote implications valid  under some additional assumptions reported in brackets.

 $$
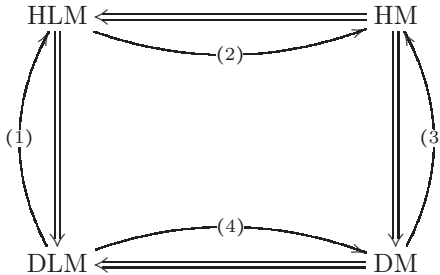
{\xymatrix@!0@C4.5cm@R=3.3cm{
 		\mbox{HLM}
 		\ar@/_1.2pc/|-{(2)
 		}[r]
 		\ar@{=>}[d]
 		&
 		\mbox{HM}
 		\ar@{=>}[l]
 		\ar@{=>}[d]\\
 		\mbox{DLM}
 		\ar@/^1.2pc/|-{(4)
 		}[r]
 		\ar@/^1.2pc/|-{(1) 	
 		} [u]
 		&
 		\mbox{DM}
 		\ar@{=>}[l]
 		\ar@/_1.2pc/|-{
 			(3)
 		} [u]
 }}$$
 (1): This implication holds true for abelian groups that are either compact (Theorem A), or connected locally compact (Theorem B).\\
 (2): This implication holds true for compact torsion-free groups (for a more general result see \cite[Theorem B]{firstpaper}). \\
 (3):  This implication holds true for locally compact abelian groups (Theorem C), but fails even for compact two-step nilpotent groups (Example \ref{ex:ac}).\\
 (4): This implication holds true for compact groups having no finite normal non-trivial subgroups,
 that are either abelian (\cite[Corollary 4.7]{ACDD}) or totally disconnected (\cite[Proposition 3.11]{firstpaper}).\\

\section{Notation and preliminary results} \label{sec:pr}
We denote by $\Z$  the group of integers, and by  $\N$ and $\N_{+}$  its subsets of non-negative integers and positive natural numbers, respectively. The groups of rationals, reals and the unit circle  are denoted, respectively, by $\mathbb{Q}, \mathbb{R}$  and  $\mathbb{T}.$ For $n\in \N_+$, we denote by $\Z(n)$ the finite cyclic group with $n$ elements. For a prime number $p$, $\Q_p$ stands for the field of $p$-adic numbers, $\Z_p$ is its subring of $p$-adic integers and $\mathbb{Z}(p^\infty)$ is the  quasicyclic $p$-group (the Pr\" ufer group).

Let $G$ be a group. The abbreviation $K\leq G$ is used to denote a  subgroup  $K$ of $G$ and $e$ denotes the identity element.
If $A$ is a non-empty subset of $G$, we denote by $\langle A\rangle$ the subgroup of $G$ generated by $A$. In particular, if $x$ is an element of $G$, then $\langle x\rangle$ is a cyclic subgroup. If $F=\langle x\rangle$ is finite, then $x$ is called a \emph{torsion} element. We denote by $t(G)$ the torsion part of a group $G$, and $G$ is called \emph{torsion} if $t(G) = G$, while $G$ is called \emph{torsion-free} if $t(G)$ is  trivial. If $G$ is abelian, then $t(G) \leq G$. For an abelian group $G$ and $n\in \N_+$, let $nG=\{ng: g\in G\}$. Then $G$ is \emph{bounded} if $n G$ is trivial for some $n\in \N_+$, and a bounded group is torsion.
If $G=nG$ for every $n\in \N_+$, then $G$ is called \emph{divisible}.

All  topological groups in this paper are assumed to be Hausdorff. The closure of $H$ in $G$ is denoted by $\overline{H},$ and $c(G)$ is the connected component of $G$. The {\em weight} and {\em density} of $G$ are denoted by $w(G)$ and $d(G)$, respectively. The Pontryagin dual of a locally compact abelian group $G$ is denoted by $\widehat{G}$.

All unexplained terms related to general topology can be found in \cite{En}. For background on abelian groups, see \cite{Fuc}.

\bigskip

There exist useful criteria for establishing the minimality, local minimality, and total minimality of a dense subgroup of a minimal, locally minimal, and totally minimal group, respectively. These criteria are based on the following definitions.
\begin{definition}
Let $H$ be a subgroup of a topological group $G$.
\ben
	\item \cite{P,S71} $H$ is  {\it essential} in $G$ if $H\cap N\neq \{e\}$ for every non-trivial closed normal subgroup $N$ of $G.$
	\item \cite{ACDD} $H$ is  {\it locally essential} in $G$ if there exists a neighborhood $V$  of $e$ in $G$ such that $H\cap N\neq \{e\}$ for every non-trivial closed normal subgroup $N$ of $G$ which is contained in $V.$
\item \cite{ts} $H$ is {\it totally dense} in $G$ if $H\cap N$ is dense in $N$ for every closed normal subgroup $N$ of $G$.
\een
\end{definition}

Obviously, a totally dense subgroup of $G$ is dense and essential in $G$, and an essential subgroup  is locally essential.

\begin{fact}\label{Crit}
Let $H$ be a dense subgroup of a topological group $G.$
\ben
\item  \cite[Minimality Criterion]{B}  $H$ is minimal if and only if $G$ is minimal and $H$ is essential in $G$ (for compact $G$ see also \cite{P,S71}).
\item  \cite [Local Minimality Criterion]{ACDD}  $H$ is locally  minimal if and only if $G$ is locally minimal and $H$ is locally essential in $G.$
\item  \cite[Total Minimality Criterion]{DP}
$H$ is totally minimal if and only if $G$ is totally minimal and $H$ is totally dense  in $G$.
\een
\end{fact}

By the Minimality  Criterion, a topological group is densely minimal if and only if it is minimal, and every dense subgroup is  essential. Similarly, by the Local Minimality Criterion, a topological group is densely locally minimal if and only if it is locally minimal, and every dense subgroup is locally  essential.

The following easy lemma is used in our main theorems.

\begin{lemma}\label{lem:DHM}
If a direct product is \D(L)M, then each of its factors is \D(L)M.	
\end{lemma}
\begin{proof}
Let $K_i$, $i \in I$, be topological groups, and assume $K= \displaystyle{\prod}_{i\in I}K_{i}$ is \D(L)M.
Fix an index $i\in I$, and let $H_{i}$ be a dense subgroup of $K_{i}$. Then $H= H_{i} \times \displaystyle{\prod}_{i\neq j\in I}K_{j}$ is a dense subgroup of $K$, so $H$ is (locally) minimal. Thus $H_{i}$ is (locally) minimal itself.
\end{proof}

Using Example \ref{ex:DHM:groups}(2)-(3) and Lemma \ref{lem:DHM} we obtain the following.
\begin{corollary}\label{cor:product} Let $G=\displaystyle{\prod}_p\Z_{p}^{\kappa_{p}}$ be infinite, where $\kappa_p$ is a cardinal for every prime $p$. Then $G$ is densely locally minimal if and only if $G=\Z_p$ for some prime $p$.
\end{corollary}

The following lemma generalizes \cite[Lemma 3.3]{DHXX}, in which only the case when $k$ is a prime number and $\alpha = \omega$ is considered.
Recall that a topological group is NSS if it has a neighborhood of the identity that does not contain  non-trivial subgroups.
\begin{lemma}\label{direct:sum:cyclic}
	If $k$ is a positive integer, and $\alpha$ is an infinite cardinal, then the compact group  $\Z(k)^{\alpha}$ is not densely locally minimal. In particular, 	an infinite compact bounded abelian group is not densely locally minimal.
\end{lemma}
\begin{proof}
	We show that the direct sum $\Z(k)^{(\alpha)}$  is not locally  minimal. Consider any countable partition of $\alpha$ into infinite subsets, $\alpha = \displaystyle{\bigcup}_{i \in \N}P_i$ with infinite $P_i$'s. For each $i\in \N$, let $D_i$ be the diagonal of $\Z(k)^{P_i}$. Then $\displaystyle{\prod}_{i\in \N} D_i$ is a closed subgroup of $\displaystyle{\prod}_{i\in \N}\Z(k)^{P_i}=\Z(k)^{\alpha}$ that is not NSS and that trivially meets $\Z(k)^{(\alpha)}$.  Hence $\Z(k)^{(\alpha)}$ fails to be locally minimal by  \cite[Lemma 2.10]{DHXX}.
	
	For the second assertion, let $G$ be an infinite bounded abelian group. Then $G$ has the form $G= \displaystyle{\prod} _{i=1}^{n} \Z(k_{i})^{\alpha_{i}}$ for positive integers $n$, $k_1, \ldots, k_n$, and cardinal numbers $\alpha_1, \ldots, \alpha_n$,
	with, say, $\alpha_1$ infinite. Then $\Z(k_{1})^{\alpha_{1}}$ is not \DLM\ by the first part of the proof, so that $G$ is not \DLM\ by Lemma \ref{lem:DHM}.
\end{proof}

\section{Proof of Theorem A}\label{sec:thmA}

In this section we prove Theorem A and Theorem B, while Theorem C is proved in Section \ref{sec:vs}.

According to \cite{DAGB},   a topological abelian group $G$ is \emph{w-divisible} if $w(G)= w(mG)\geq \omega$ for every  $m\in\N_{+}$.
One of the key ingredients in the proof of Theorem A and Theorem C is the following splitting theorem for compact abelian groups.
\begin{fact} \label{fac:SPL}\cite[Theorem 1.3]{DAG2}
	Let $K$ be a compact abelian group.
	Then $K$ splits topologically in a direct product $K=K_{tor}\times K_{d}$,
	where $K_{tor}$ is a compact bounded abelian group, while the compact abelian group $K_{d}$ is w-divisible.
\end{fact}

The dual of the above result, namely a factorization theorem for discrete abelian groups, is proved in \cite{GaMa}.

Recall that a topological group is {\em linearly topologized} if it has a local base at the identity consisting of open subgroups (the term \emph{non-archimedean} is also used by some authors).
For example, profinite groups are linearly topologized, and indeed in Theorem A  we apply the following fact to profinite groups.

\begin{fact}\cite[Proposition 3.4(c)]{ACDD}\label{fac:lintop}
	Let $G$ be a topological abelian group and let $H$ be a subgroup of $G$. If $H$ is locally essential in $G$, then for every closed and linearly topologized subgroup $N$ of $G$,  there exists an open subgroup $V$ of $N$ such that $H\cap V$ is essential in $V$.
\end{fact}

The following results which appear in \cite{DAG} are frequently used in the sequel.
\begin{fact} \label{fac:WDC}  \
	\ben
	\item	If $K$ is a w-divisible compact abelian group, then $K$
	admits a dense free abelian subgroup $F$ with $|F| = d(K)$.
	\item Let $N$ be a quotient of $G$ with $h:G\twoheadrightarrow N$ the canonical projection. If $H$ is a subgroup of $G$ such that $h(H)$ is dense in $N$ and $H$ contains a dense subgroup of $\ker h$, then $H$ is dense in $G.$
	\een
\end{fact}

The following lemma is used in the subsequent Proposition \ref{new:W:L}.
\begin{lemma}\label{new:lemma:W:L}
Let $G$ be 
a topological abelian  group, 
and $\kappa$ be a cardinal such that:
\begin{enumerate}
\item $d(G) \leq \kappa \cdot \omega$;
\item $G$ has a subgroup $R \cong \Z_p^\kappa$ for some prime number $p$.
\end{enumerate}
If $G$ is densely locally minimal, then $\kappa$ is finite. In particular, $G$ is separable.
\end{lemma}
\begin{proof}
%
Assume by contradiction $\kappa$ to be infinite, and let $H$ be a dense subgroup of $G$ with $|H| = d(G) \leq \kappa \cdot \omega = \kappa$. Then $H$ is locally minimal, so it is locally essential in $G$ by the Local Minimality Criterion.  Using  Fact \ref{fac:lintop}, we find an open subgroup $U$ of $R$ (necessarily isomorphic to $\Z_p^\kappa$), such that $H \cap U$ is essential  in $U$ and clearly $|H \cap U|\leq |H|\leq \kappa$.

As $\kappa$ is infinite, the group $\Z_p^\kappa$ has a family of size $2^\kappa$ of closed subgroups isomorphic to $\Z_p$ with pairwise trivial intersection (see \cite{DPS}). The essentiality of $H \cap U$ in $U$ implies that $|H \cap U|\geq 2^\kappa$, a contradiction.
\end{proof}

In the next proposition, we consider some conditions guaranteeing that a quotient of a densely locally minimal compact abelian group is a finite power of $\T$.

\begin{proposition}\label{new:W:L}
	Let $G$ be a compact abelian group, $N$ be a closed subgroup of $G$, and $\kappa$ be a cardinal such that  $d(N) \leq \kappa \cdot \omega$ and $G/N \cong \T^{\kappa}$. If $G$ is densely locally minimal, then $\kappa$ is finite.
\end{proposition}
\begin{proof}
It suffices to check that $G$ satisfies conditions (1) and (2) of Lemma \ref{new:lemma:W:L}, when $\kappa$ is infinite. So let $k = \kappa \cdot \omega$ be infinite, and let $q:G\to G/N$ be the canonical projection.

It is easy to see that $\T^\kappa$ is w-divisible, so there exists a dense free subgroup $B$ of $\T^\kappa$ with $|B|\leq\kappa$ by Fact \ref{fac:WDC}(1).
As $B$ is free, there exists a free subgroup $B_1$ of $G$ such that $q(B_1) = B$ and $|B_1| = |B|\leq \kappa$. Let $D$ be a dense subgroup of $N$ of minimal cardinality $d(N) \leq \kappa$. Then $H= B_1+ D$ is dense in $G$ by Fact \ref{fac:WDC}(2), and $|H| \leq\kappa$. In particular, $d(G)\leq \kappa$.

We now prove that condition (2) is satisfied. First note that $ \T^\kappa$ contains $\Z_p^\kappa$ for every prime $p$. For the compact subgroup $A= q^{-1}(L)$ we have $A/N \cong L\cong \Z_p^\kappa$, that  is torsion-free. Being divisible, the Pontryagin dual $\widehat{A/N}$ is topologically isomorphic to a direct summand of $\widehat{A}$ (see \cite[Theorem 21.2]{Fuc}). So $\widehat{ \widehat{A/N} } \cong A/N \cong \Z_p^\kappa$ is (topologically isomorphic to) a subgroup of $A$, hence of $G$.
%
%
%
\end{proof}

The following result is folklore.
\begin{fact} \label{fac:CON}
	If $f: G\to H$ is a continuous surjective homomorphism of compact groups, then $f(c(G))= c(H)$.
\end{fact}
Now we apply Proposition \ref{new:W:L} to a particular instance of Theorem A, where we prove that condition (b) implies condition (c), under some additional assumptions.
\begin{proposition}\label{lem:STEP4}
	Let $K$ be a  compact abelian group and assume that $K$ has a subgroup $N\cong F\times \Z_p$, where $F$ is finite and $p$ is a prime number, and such that $K/N \cong \T^\kappa$ for some cardinal  $\kappa$.
	
	If $K$ is densely locally minimal, then $\kappa = 0$ (i.e., $K = N$), so in particular $K$ has an open subgroup isomorphic to $\Z_p$.
\end{proposition}
\begin{proof}
	Let $q: K \to K/N \cong \T^\kappa$ be the quotient homomorphism with $\ker q = N$. As $d(N)=\omega$, $\kappa$ is finite by Proposition \ref{new:W:L}. Assume by contradiction that $\kappa \neq 0$, so let $\kappa=n\in \N_+.$
	
By Fact \ref{fac:CON},
\begin{equation*}	
q(c(K)) = c(\T^n) = \T^n.
\end{equation*}
	
	Since both $\T^n$ and $F \times \Z_p$ are metrizable, the group $K$ is also metrizable (see \cite[Corollary 3.3.20]{AT}), and so is $c(K)$.
	It is known that a compact metrizable connected abelian group is monothetic (see \cite[Example 8.75]{HM}),  i.e., it contains a dense cyclic subgroup (that is necessarily infinite, if the group is non-trivial). Fix a dense (infinite) cyclic subgroup $B$ of $c(K)$ and call $b$ its generator. Moreover, fix also a cyclic subgroup $\langle z \rangle$ dense in $\Z_p$. Now let $x:=z + b$ and $H = F + \langle x \rangle$.\medskip
	
\noindent
\textbf{Claim 1} $H$ is dense in $K$.
\begin{proof} Let $A$ be the closure of $H$ in $K$. In order to prove that  $A = K$ we first prove that  $q(A) = q(K) = \T^n$.
Since $q(x) = q(b)$  and $q(F) = 0$, we deduce that  $q(H) = \langle q(b) \rangle$. As  $\langle b \rangle$ is dense in $c(K)$ and $q(c(K)) = \T^n$, it follows that $q(H)$  is dense in  $\T^n$.
So, $q(A) = \overline{q(H)} = \T^n$ and indeed $q(c(A)) = \T^n$ by Fact \ref{fac:CON}.
		
Next we show that $c(K)$ is contained in $A$. By a dimension theorem given in \cite{PA}, we obtain \[\dim c(A) =\dim c(K)=\dim \T^n=n,\] as $\ker q=N$ is zero dimensional. Moreover, since the quotient group $c(K)/c(A)$ is connected and
\[
\dim (c(K)/c(A)) = \dim c(K) - \dim c(A) = 0,
\]
we conclude that $c(K) = c(A)$ is contained in $A$.
		
To end the proof of the claim, i.e., the proof of the equality $A = K$, it suffices to see that $A + c(K)=K$. Consider the quotient map $s : K \to K/c(K) \cong N/(c(K) \cap N)$, as $K = c(K) + N$. Since $b\in c(K)$, $s(b)=0$, so $s(x) = s(z)$. Then $s(H) = s(F + \langle z \rangle)$. Therefore $s(A) = s(\overline{H})= \overline{s(H)} =\overline{s(F + \langle z \rangle)}= s(\overline{F + \langle z \rangle}) =  s(N)$, as $\overline{F + \langle z \rangle} = N$. Thus $s(A) = K/c(K)$, it follows that $A + c(K) = K$.
\end{proof}
	
\noindent
\textbf{Claim 2.} $H$ is not locally essential in $K$.
\begin{proof}
Assume $H$ to be locally essential in $K$ and let $U$ be an open neighborhood of $e$ in $K$ with compact closure, and witnessing the local essentiality of $H$. As $\displaystyle{\bigcap}_{n \in \N} p^n \Z_p = \{e\} \subseteq U$, by the compactness of $\overline{U}$ we find  $n\in \N $ such that $p^n \Z_p \cap \overline{U}  \subseteq U$. Then $p^n \Z_p \cap H$ is non-trivial, so it is infinite and $H \cap N$ is infinite as well. Since $F\leq N$, by the modular law we obtain $H \cap N= (F+\langle x \rangle)\cap N= F+(\langle x \rangle \cap N)$. In particular, also  $\langle x \rangle \cap N$  is infinite so has finite index in $\langle x \rangle$. Furthermore, $|H/H\cap N|=|(F+\langle x \rangle)/(F+(\langle x \rangle \cap N))|\leq |\langle x \rangle /\langle x \rangle \cap N|<\infty.$ It follows that $q(H) =(H + N)/N \cong H/H\cap N$ is finite. This contradicts the fact that $q(H)$ is dense in $\T^n = K/N$.
\end{proof}
	
Claim 1 and Claim 2 complete the proof of Proposition \ref{lem:STEP4}, contradicting the assumption that $K$ is densely locally minimal.
\end{proof}

{\bf Proof of Theorem A.} We have to prove  that for a compact abelian group $K$, the following conditions are equivalent:
\ben[(a)]
\item $K$ is a \HLM \ group;
\item $K$ is a \DLM \ group;
\item $K$ either is a Lie group or $K$ has an open subgroup isomorphic to $\Z_p$ for some prime $p$.
\een

 $(a)\Leftrightarrow(c)$: Follows from Fact \ref{fac:ext}.\\
$(a)\Rightarrow(b)$: It is clear by the definition.\\
$(b)\Rightarrow(c)$: Write $K=K_{tor}\times K_{d}$, where $K_{tor}$ is a compact bounded abelian group and $K_{d}$ is w-divisible, as in Fact \ref{fac:SPL}. By Lemma \ref{lem:DHM}, both $K_{tor}$ and $K_d$ are \DLM,
so $K_{tor}$ is finite by Lemma \ref{direct:sum:cyclic}.
In the rest of the proof, we show that $G = K_d$ either is a Lie group or has an open subgroup isomorphic to $\Z_p$ for some prime $p$, so that $K$ satisfies condition (c).

As $G$ is a compact abelian group, there exists a closed profinite subgroup $N$ of $G$ such that  $G/N \cong \T^\kappa$ for some cardinal $\kappa$ (see \cite[Proposition 8.15]{HM}). Let $q:G\to G/N$ be the canonical projection.

We break the proof into three steps.

\medskip
\noindent
\textbf{Step 1.}
By our assumptions on $G$ and using Fact \ref{fac:WDC}(1), there exists a dense free subgroup $F$ of $G$. As $G$ is \DLM, we deduce that $F$ is locally essential by the Local Minimality Criterion.
Since $N$ is profinite,
Fact \ref{fac:lintop} provides an open subgroup $O$ of $N$ (so of finite index), such that $F \cap O$ is essential in $O$. As $F \cap O$ is torsion-free, this yields that $O$ is torsion-free. Thus $O \cap t(N) = \{0\}$, and as $[N:O]$ is finite we deduce that $t(N)$ is finite. So, by \cite[Theorem 27.5]{Fuc}, we obtain $N \cong N_1 \times t(N)$ with $N_1$ torsion-free.

\medskip
\noindent
\textbf{Step 2.} Assume that $N$ is finite with $|N|=m$. If $\kappa$ is infinite, then clearly $d(N)=m\leq \kappa$, contradicting Proposition \ref{new:W:L}.

As $\kappa$ is finite, Pontryagin duality implies that $N\cong \widehat{N}\cong \widehat{G}/\Z^\kappa$. Using the three-space property of finitely generated groups, we deduce that $\widehat{G}$ is finitely generated. By \cite[Theorem 15.5]{Fuc}, a (discrete) finitely generated abelian group has the form $A \times \Z^n$, where $A$ is finite and $n\in \N$, so $G \cong \widehat{\widehat{G}} \cong A \times \T^n$ is a Lie group.

\medskip
\noindent
\textbf{Step 3.} We are left now with the case when $N= N_1\times t(N)$ is infinite, i.e., $N_1$ is infinite. As $N_1$ is compact, totally disconnected and torsion-free, the dual of $N_1$ is discrete, torsion and divisible. Thus $\widehat{N_1}\cong \displaystyle{\bigoplus}_p\Z(p^\infty)^{(\kappa_p)}$
and $N_1 \cong \displaystyle{\prod}_p\Z_{p}^{\kappa_{p}}$. We show that $N_1$ is  isomorphic to $\Z_p$  for some prime $p$.

Otherwise, $N_1$ is not densely locally minimal, by Corollary \ref{cor:product}. According to  Lemma \ref{lem:DHM}, also $N$ has a dense subgroup $T$ that is not locally essential. Since $G/N\cong \T^\kappa$ is w-divisible, Fact \ref{fac:WDC}(1) provides a dense free subgroup $F_1$ of $\T^\kappa$ with $|F_1|\leq\kappa\cdot \omega$, and we can choose a free subgroup $F_2$ of $G$ that trivially meets $N$ and that is sent onto $F_1$ by the quotient map $q$. Consider the subgroup $K = T + F_2$ of $G$.
Since $T$ is a dense subgroup of $N$ and $q(K)$ is dense in $\T^\kappa$, Fact \ref{fac:WDC}(2) implies that  $K$ is dense in $G$.
Using the fact that $T$ is not locally essential in $N$, and that $F_2 \cap N$ is trivial, one can see that $K$ is not locally essential in $G$, a contradiction.

Then, $N_1\cong\Z_p$ and $N \cong \Z_p \times t(N)$ with $t(N)$ finite.   Then $\kappa=0$ by Proposition \ref{lem:STEP4}, i.e., $G = N$ has an open subgroup isomorphic to $\Z_p$.
\qed
\medskip

As a corollary of Theorem A we obtain the following.
\begin{corollary}\label{lem:SHM}
If a compact torsion-free abelian group $K$ is \DLM, then $K$ is isomorphic to $\Z_{p}$ for some prime $p$.
\end{corollary}
\begin{proof}
By our assumption on $K$, we obtain that $K$ is \HLM\ according to Theorem A. Moreover, since $K$ is also torsion-free, we deduce that it is \HM\ by \cite[Theorem B]{firstpaper}. Hence $K\cong \Z_p$ for some prime $p$ by Fact \ref{TeoP}.
\end{proof}

By Theorem A, a compact abelian group is \DLM \ precisely when it is \HLM. In Example \ref{ex:ac} we show, in contrast, that there exist compact two-step nilpotent groups that are \DM, but not \HLM.

The following easy result appears in \cite{firstpaper}. We apply it in  Corollary \ref{cor:center} to two-step nilpotent topological groups i.e., to groups $G$ satisfying $G'\leq Z(G)$, where  $Z(G)$ is the center  and $G'$ is the derived subgroup of $G$.

\begin{fact}
Let $G$ be a group and $G'$ be its derived subgroup.  If $H$ is a subgroup of $G$ that is either
 non-central normal or non-abelian, then $H\cap G'$ is non-trivial.
\end{fact}

\begin{corollary}\label{cor:center}
Let $G$ be a two-step nilpotent topological group.
\ben
\item The center $Z(G)$  is essential in $G$.
\item If $H$ is a non-abelian subgroup of $G$, then $H\cap Z(G)$ is non-trivial.
\een
\end{corollary}

\begin{example}\label{ex:ac}
Consider the ring multiplication $w:\Z_p\times \Z_p\to \Z_p, \ w(a,b)=ab$ and let $G=H(w)=(\Z_p\times \Z_p) \rtimes_{\alpha}\Z_p$ be the induced generalized Heisenberg group (see \cite{DMe}). The group $G$ is isomorphic to the following subgroup of the special linear group $\SL(3, \Z_p)$
 \[
\Bigg\{\left(\begin{array}{ccc}
1 & a &  b \\
0 & 1 & c \\
0 & 0 & 1
\end{array}\right)\bigg | \ \ a,b,c \in \Z_p\Bigg\}.
\]
It is known that $G$ is a compact two-step nilpotent group, and  $Z(G)=(\Z_p\times \{0\}) \rtimes\{0\}\cong \Z_p$. As $G$ contains a copy of $\Z_p\times \Z_p$, we deduce by Fact \ref{fac:ext} that $G$ is not \HLM.

Now we show that $G$ is \DM, so let $H$ be a dense subgroup of $G$. In view of the Minimality Criterion, we have to prove that $H$ is essential in $G$. So let $N$ be a non-trivial closed normal subgroup of $G$, and we show that $H\cap N$ is non-trivial. By Corollary \ref{cor:center}(1), the subgroup $N_1 = N \cap Z(G)$ is non-trivial, and it is closed in $Z(G)$. Clearly, the dense subgroup $H$ of $G$ is non-abelian, so $H_1=H\cap Z(G)$ is non-trivial by Corollary \ref{cor:center}(2).
As every non-trivial subgroup of $\Q_p$ is essential (see \cite[Lemma 2.5]{firstpaper}), and $Z(G)$ is closed in $\Q_p$, it follows that every non-trivial subgroup of $Z(G)$ is essential in $Z(G)$, and in particular $H_1$ is essential in $Z(G)$. As $N_1$ is closed in $Z(G)$, this implies that $H_1\cap N_1$ is non-trivial. In particular, $H\cap N$ is non-trivial.
\end{example}

\bigskip

\noindent
{\bf Proof of Theorem B.} We just need to prove $(b)\Rightarrow(c)$, i.e., if a connected locally compact abelian group $G$ is \DLM, then it is a Lie group.

Assume that $G$ is a connected locally compact abelian group. It is known that $G \cong \R^n \times K$, where $n\in \N$ and $K$ is a connected compact group. Then $K$ is \DLM\ by Lemma \ref{lem:DHM}; moreover, $K$ is connected, so it is a Lie group by Theorem A.
Thus $G$ itself is a Lie group. \qed

\section{Dense minimality vs Hereditary minimality}\label{sec:vs}

We start this section proving another sufficient condition for a compact abelian group to be isomorphic to some $\Z_p$.
\begin{lemma}\label{lem:MTF}
	If a w-divisible compact abelian group $K$ is \DM, then it is isomorphic to $\Z_p$ for some prime $p$.
\end{lemma}
\begin{proof}
By Fact \ref{fac:WDC}(1), $K$ has a dense free abelian subgroup $H$.
Then $H$ is essential in $K$ by the Minimality Criterion, and $H$ is torsion-free, so $K$ is torsion-free as well. Then $K \cong \Z_{p}$ for some prime $p$ by Corollary \ref{lem:SHM}.
\end{proof}

The following result is contained in the proof of \cite[Theorem 1.3]{DHXX}, but we give a proof of this property for the sake of the reader. Together with the above Lemma \ref{lem:MTF}, and the splitting theorem  Fact \ref{fac:SPL}, it will provide the proof of Theorem C.
\begin{lemma}\cite{DHXX}		\label{Zp:times:finite:ab}
	Let $p$ be a prime number, and $F$ be a finite abelian group such that $K=\Z_p \times F$ is \DM. Then $F$ is trivial.
\end{lemma}
\begin{proof}
	By contradiction, let $F= \{c_{1}, c_{2},\ldots, c_{n}\}$ be non-trivial. Pick elements $\xi_{1}, \xi_{2}, \ldots , \xi_{n}$ in $\Z_p$ independent with $1$, and let $H$ be the subgroup of $K$ generated by $\Z \times \{0\}$ and $\{(\xi_{i}, c_{i}) : i=1, 2, \ldots, n\}$.
	
Considering the projection $p : K\to F$, we deduce by Fact \ref{fac:WDC}(2) that $H$ is a dense subgroup of $K$, hence essential by the Minimality Criterion.  An element $h\in H$ has the form $h=(k,0)+ \sum_{i=1}^{n}k_{i} (\xi_{i}, c_{i})$, so the closed nontrivial subgroup $ \{0\}\times F$ of $K$ trivially meets $H$, a contradiction.
\end{proof}

Now we are in position to prove Theorem C.

\bigskip

\noindent
{\bf Proof of Theorem C.} We have to prove that for an infinite locally compact abelian group $K$, the following conditions are equivalent:
\ben[(a)]
\item $K$ is a \HM \ group;
\item $K$ is a \DM \ group;
\item $K$ is isomorphic to $\Z_p$ for some prime $p$.\\
By Fact \ref{fact:step}, we may assume that $K$ is compact.
\een$(a)\Leftrightarrow(c)$ follows from Fact \ref{fac:HMZ}.\\
$(a)\Rightarrow(b)$: It is clear by the definition.\\
$(b)\Rightarrow(c)$:
Write $K=K_{tor}\times K_{d}$, where $K_{tor}$ is a compact bounded abelian group and $K_{d}$ is w-divisible, as in Fact \ref{fac:SPL}. By Lemma \ref{lem:DHM}, both $K_{tor}$ and $K_d$ are \DM, so $K_{tor}$ is finite by Lemma \ref{direct:sum:cyclic}, while $K_d$ is isomorphic to $\Z_p$ by Lemma \ref{lem:MTF}. Now apply Lemma \ref{Zp:times:finite:ab}.
\qed

\medskip

By Theorem C, a locally compact abelian group is \DM \ precisely when it is \HM, and these properties characterize the $p$-adic integers among the infinite locally compact abelian groups. Example \ref{ex:ac} shows that the equivalence between those two properties need not hold relaxing the abelianity to two-step nilpotency, even for compact groups.

The next example shows that a compact metabelian group which is simultaneously \HLM \ and \DM\  need not be \HM. In what follows, $A^*$ denotes the multiplicative group of the invertible elements of a ring $A$.
\begin{example}\label{DHM:notHM}
Let $G=(\Q_p,+)\rtimes \Q_p^*$ and consider its compact subgroup $H=\Z_p \rtimes \Z_p^*$.
By \cite[Theorem A]{firstpaper}, the group $G$ (in particular, also its subgroup $H$)  is \HLM.
By \cite[Corollary 2.11]{firstpaper}, both $G$ and $H$ are also \DM. Since $\Z_p^*\cong \Z_p\times F$, where $F$ is a finite non-trivial group, it follows that $\Z_p^*$ is not \HM\ by Fact \ref{TeoP}. In particular, $H$ and $G$ are  not \HM.

We do not know if there exists a \HM\  locally compact group which is neither discrete nor compact (see \cite[Question 7.3]{firstpaper}). On the other hand, the group $G=(\Q_p,+)\rtimes \Q_p^*$ is an example of a \DM\  locally compact group which is neither discrete nor compact.
\end{example}

Recall that a (topologically) simple group is a group whose only normal subgroups (closed normal subgroups) are the trivial group and the group itself. Obviously, a minimal topologically simple group is totally minimal. Moreover, every dense subgroup of a topologically simple group is totally dense, so the Total Minimality Criterion implies:

\begin{lemma}\label{simple}
	A minimal topologically simple group is \DTM.
\end{lemma}

The following example shows that there exists a discrete (densely) totally minimal group which is not \HM.
\begin{example} \label{ex:nontop} A discrete minimal  group is  also called {\it non-topologizable}. The first example of a non-topologizable group was provided by Shelah \cite{Sh} under the assumption of the Continuum Hypothesis. His example is simple and torsion-free, so this group is also (densely) totally minimal, yet not hereditarily minimal, as discrete hereditarily minimal groups are torsion by \cite[Lemma 3.5]{firstpaper}.
\end{example}

A \HM\ locally compact group is totally disconnected (see \cite[Proposition 3.9]{firstpaper}), while the next example shows that there exist \DTM, compact, pathwise connected groups.

\begin{example}\label{ex:SO}
Let $n\geq 5$ be an odd number and consider the special orthogonal group $G=\SO(n, \R)$ of degree $n$ over the reals. The group $G$ is compact, and simple (see for example \cite[Theorem 7.5]{grove}), so it is \DTM\ by Lemma \ref{simple}. On the other hand, it is also known that $G$ is pathwise connected (see for example \cite{naive}).
\end{example}

\section{Open questions and concluding remarks}\label{sec:open}
By Theorem B, a densely locally minimal, connected, locally compact abelian group is Lie. Can we omit the assumption `abelian' in the hypotheses of Theorem B? In other words:
\begin{question}
Let $G$ be a \DLM\ connected locally compact group. Is $G$ a Lie group?
\end{question}

\medskip

Dikranjan and Stoyanov classified all \HM \ abelian groups.
\begin{fact}\cite{DS}\label{fac:hmabel}
	Let $G$ be a topological abelian group. Then the following conditions are equivalent:
	\ben
	\item each subgroup of $G$ is totally minimal;
	\item  $G$ is \HM;
	\item $G$ is topologically isomorphic to one of the following groups:\ben [(a)]
	\item  a subgroup of $\Z_p$ for some prime $p$,
	\item  a direct sum $\bigoplus F_p$, where for each prime $p$ the group $F_p$ is a finite abelian $p$-group,
	\item $X\times F_p$,  where $X$ is a rank-one subgroup of $\Z_p$ and  $F_p$ is a finite abelian $p$-group.
	\een \een
\end{fact}

\begin{question}
Is it possible to add in Fact \ref{fac:hmabel} also at least one of the following apparently weaker conditions:
\ben
\item [(1$^\ast$)]  $G$ is \DTM ?
\item [(2$^\ast$)]  $G$ is \DM?
\een
\end{question}

Note that if condition $(2^\ast)$ is equivalent to the other conditions in Fact \ref{fac:hmabel}, then also $(1^\ast)$ is, as the implications $(1) \rightarrow (1^\ast) \rightarrow (2^\ast)$ follow from the definitions.

\bigskip

Let $K$ be a Lie group. Recall that by the Hilbert-Smith Conjecture $\Z_p$ cannot act effectively on $K$. This means that if $\alpha:\Z_p\times K \to K$ is a continuous action, then $\ker \alpha$ is non-trivial.

\begin{fact}\label{fact:lie}\cite[Theorem 6]{K}
Every non-discrete Lie group contains a non-trivial continuous homomorphic image of $\R$.
\end{fact}

By Fact \ref{fac:ext}, a hereditarily locally minimal, locally compact group that is abelian or connected either is a Lie group or has $\Z_p$ as an open subgroup. In \cite{firstpaper} we proved that the non-abelian, totally disconnected group $\Q_p \rtimes \Q_p^*$ is hereditarily locally minimal, while it neither is a Lie group, nor has $\Z_p$ as an open subgroup. So now we consider the case when a group is hereditarily locally minimal, locally compact, not totally disconnected, under the Hilbert-Smith Conjecture.
\begin{proposition}\label{prop:hscon}
	Assuming the Hilbert-Smith Conjecture, let $G$ be a \HLM \ locally compact group.
	If $G$ is not totally disconnected, then every closed abelian subgroup of $G$ is Lie; in particular, $G$ cannot contain a copy of $\Z_p$ for any prime $p$.
\end{proposition}
\begin{proof}
Let $A$ be a closed abelian subgroup of $G$. In view of Fact \ref{fac:ext}, it suffices to prove that $A$ cannot contain a copy of $\Z_p$ for any prime $p$.
By contradiction, let $K\cong \Z_p$ be a subgroup of $A$.
	
As the connected component $c(G)$ of $G$ is non-trivial, it is a non-trivial Lie group by \cite[Theorem 1.10]{DHXX}, and we claim that $H=K\cap c(G)$ is trivial. Otherwise, it would be a closed non-trivial subgroup of $K\cong \Z_p$, so $H\cong \Z_p$ itself, but this is impossible as $H$ is a subgroup of the Lie group $c(G)$.	
	
As $c(G)$ is normal in $G$, the topological semi-direct product $c(G)\rtimes_{\alpha} K$, defined by conjugations in $G$, is (isomorphic to) a subgroup of $G$. By Hilbert-Smith Conjecture, $M=\ker \alpha$ is non-trivial, and being a closed subgroup of $K\cong \Z_p$, $M$  is isomorphic to $\Z_p$ as well.
So also
	\[
	c(G)\rtimes_{\alpha} M = c(G)\times M \cong c(G)\times \Z_p
	\]
is a subgroup of $G$, hence a \HLM \ locally compact group. Since $c(G)$ is a non-trivial connected Lie group, it must contain a non-discrete abelian Lie group $N$, by Fact \ref{fact:lie}.
It follows that $N\times M$ is \HLM, contradicting \cite[Corollary 4.7]{DHXX}.
\end{proof}

\begin{question}\label{q:HSC}
Can we replace in Proposition \ref{prop:hscon} the condition ``\HLM'' with ``\DLM'' obtaining the same conclusion?
\end{question}

\subsection*{Acknowledgments}
The first-named author takes this opportunity to thank Professor Dikranjan for his generous hospitality and support.
The second-named author  is partially supported by  grant PSD-2015-2017-DIMA-PRID-2017-DIKRANJAN PSD-2015-2017-DIMA - progetto PRID TokaDyMA
 of Udine University. The third and fourth-named authors are supported by Programma SIR 2014 by MIUR, project GADYGR, number RBSI14V2LI, cup G22I15000160008 and by INdAM - Istituto Nazionale di Alta Matematica.


\begin{thebibliography}{10}

\frenchspacing
\bibitem{AT} A. Arhangel'skii and M. Tkachenko,  \emph{Topological groups and related structures},
Vol. \textbf{1} of Atlantis Studies in Math. Series Editor: J. van Mill. Atlantis Press, World Scientific, Amsterdam-Paris,
2008.

\bibitem{ACDD} L. Au{\ss}enhofer, M.J. Chasco, D. Dikranjan and X. Dominguez, \textit{Locally Minimal Topological Groups 2}, J. Math. Anal. Appl. \textbf{380} (2011), 552--570.

\bibitem{B} B. Banaschewski, \textit{Minimal topological algebras}, Math. Ann. \textbf{211} (1974), 107--114.


\bibitem{HBSTT} W.W. Comfort, \textit{Topological groups, in: K. Kunen and J. Vaughan, eds., Handbook of Set-Theoretic Topology}, North-Holland, Amsterdam. (1984), 1143--1263.


\bibitem{CD0} W. W. Comfort and D. Dikranjan, \textit{On the poset of totally dense subgroups of compact groups}, Topol. Proc. \textbf{24} (1999), 103--128.


\bibitem{DAGB}
D. Dikranjan and A. Giordano Bruno,
 \textit{w-Divisible groups},
Topol. Appl. \textbf{155} (2008), 252--272.

\bibitem{DAG}
D. Dikranjan and A. Giordano Bruno,
 \textit{Compact groups with a dense free abelian subgroup},
Rend. Istit. Mat. Univ. Trieste \textbf{45} (2013), 137--150.

\bibitem{DAG2}
D. Dikranjan and A. Giordano Bruno,
 \textit{A factorization theorem for topological abelian groups},
 Comm. Algebra. \textbf{41.1} (2015), 212--224.

\bibitem{DHXX} D. Dikranjan, W. He, Z. Xiao and W. Xi,
 \textit{Locally Compact Groups and Locally Minimal Group Topologies},
Fundam. Math. 2018, to appear.


\bibitem{DMe}
D. Dikranjan and M. Megrelishvili,
 \textit{Minimality conditions in topological groups},
Recent progress in general topology III, Atlantis Press, Paris, (2014), 229--327.


\bibitem{DP} D. Dikranjan and Iv. Prodanov, \textit{Totally minimal groups}, Annuaire Univ. Sofia Fat. Math.
M\'{e}c. \textbf{69} (1974/75) 5--11.

\bibitem{DPS} D. Dikranjan, Iv. Prodanov and  L. Stoyanov, \textit{Topological Groups: Character, Dualities and Minimal Group Topologies}, Pure and Applied Mathematics, Vol. \textbf{130}, Marcel Dekker Inc., New York-Basel, 1989.


\bibitem{DS} D. Dikranjan and L. Stoyanov,  \textit{Criterion for minimality of all subgroups of a topological
abelian group}, C. R. Acad. Bulgare Sci. \textbf{34} (1981), 635--638.

\bibitem{Doitch}
D. Do\"{\i}tchinov, \textit{Produits de groupes topologiques minimaux}, Bull. Sci. Math. \textbf{97.2} (1972), 59--64.

\bibitem{EDS}
V. Eberhardt, S. Dierolf and U. Schwanengel, \textit{On products of two (totally) minimal
	topological groups and the three-space-problem}, Math. Ann. \textbf{251} (1980) 123--128.

\bibitem{En}
R.~Engelking,
 \textit{General Topology},
Sigma Series in Pure Math, Vol. \textbf{6}, Heldermann, Berlin, 1989.

\bibitem{Fuc}
L.~Fuchs,
 \textit{Infinite Abelian groups}, Vol. I,
Academic  Press,  New York, 1970.

\bibitem{GaMa}
J. Galindo and S. Macario, \textit{Pseudocompact group topologies with no infinite compact subsets}, J. Pure and Appl. Algebra \textbf{215} (2011) 655--663.

\bibitem{grove} L. C. Grove,  \textit{Classical groups and geometric algebra}, Amer. Math. Soc.,  2001.


\bibitem{HM}  K. Hofmann and S. Morris, {\em The structure of compact groups.  A primer for the student---a handbook for the expert},
Gruyter Studies in Mathematics \textbf{25}, Walter de Gruyter \& Co., Berlin, 1998.


\bibitem{K}
I. Kaplansky,  \textit{Lie Algebras and Locally Compact Groups}, Chicago Lectures in Mathematics, Univ. of Chicago Press, 1971.

\bibitem{KH} M. A. Khan,  \textit{Chain conditions on subgroups of LCA groups}, Pacific J. Math. \textbf{86} (1980), 517--534


\bibitem{MP}
S. Morris and  V. Pestov,
 \textit{On Lie groups in varieties of topological groups},
Colloq. Math. \textbf{78} (1998), 39--47.

\bibitem{PA}
B.A. Pasynkov , \textit{On coincidence of different definitions of a dimension for locally compact groups}, (in Russian), Dokl. AN SSSR \textbf{132} (1960), 1035--1037.

\bibitem{P}
Iv. Prodanov,
 \textit{Precompact minimal group topologies and $p$-adic numbers},
Annuaire Univ. Sofia Fac. Math. M\' ec.  \textbf{66} (1971/72), 249--266.


\bibitem{Sh} S. Shelah, \textit{On a problem of Kurosh, J\' onsson groups, and applications},  In: S. I. Adian, W. W. Boone and G. Higman, Eds., Word Problems II, North-Holland, Amsterdam, (1980) 373--394.

\bibitem{ts} T. Soundararajan, \textit{Totally dense subgroups of topological groups},  \textit{ in General Topology and Its Relations to Modern Analysis and Algebra III}, Proceedings  Kanpur Topological Conference, 1968, Academia Press, Prague, 1971.


\bibitem{S71}
R.M. Stephenson, Jr., \textit{Minimal topological groups}, Math. Ann. \textbf{192} (1971), 193--195.

\bibitem{naive} J. Stillwell, \textit{Naive Lie Theory}, Springer-Verlag, New York, 2008.



\bibitem{ed} L. Stoyanov,  \textit{A property of precompact minimal abelian groups}, Annuaire Univ. Sofia Fac. Math. M\'{e}c, \textbf{70} (1975/1976), 253--260.

\bibitem{firstpaper} W. Xi, D. Dikranjan, M. Shlossberg and D. Toller,
\emph{Hereditarily minimal topological groups}, submitted.


\end{thebibliography}
\end{document}